\documentclass[fleqn,a4paper,12pt]{article}
\usepackage{amsmath,amsfonts,calrsfs,amssymb,color}
\usepackage{amsthm}

\newtheorem{Theorem}{Theorem} 

\newtheorem{Proposition}[Theorem]{Proposition}
\newtheorem{Lemma}[Theorem]{Lemma}
\newtheorem{Corollary}[Theorem]{Corollary}

\makeatletter

\def\R{\mathbb R}
\def\N{\mathbb N}

\def\E{\mathbb E}

\def\ds{\displaystyle}

\title{Estimate for $P_tD$ for the stochastic Burgers equation}

\author{Giuseppe Da Prato  \thanks{Giuseppe Da Prato, Scuola Normale Superiore, 56126, Pisa, 
Italy.   e-mail:   giuseppe. daprato@sns.it } and Arnaud Debussche \thanks{
Arnaud Debussche, IRMAR and \'Ecole Normale Sup\'erieure de Rennes, Campus de Ker Lann, 37170 Bruz, France. e-mail:arnaud.debussche@ens-rennes.fr}}

\begin{document}

\maketitle

\begin{abstract}
We consider the Burgers equation on $H=L^2(0,1)$ perturbed by white noise and the corresponding transition semigroup $P_t$. We prove a new formula for  $P_tD\varphi$ (where  $\varphi:H\to\R$ is bounded and Borel) which  depends on $\varphi$ but not on its derivative. 
Then we deduce some new consequences   for the invariant measure  $\nu$ of $P_t$ as its  Fomin differentiability and an  integration by parts formula which generalises the classical one  for gaussian measures.

\end{abstract}

\noindent {\bf 2000 Mathematics Subject Classification AMS}: 60H15, 35R15\medskip


\section{Introduction}
We consider the following stochastic Burgers equation in the interval $[0,1]$ with Dirichlet boundary conditions,
\begin{equation}
\label{e0.1}
\left\{\begin{array}{lll}
dX(t,\xi)&=&(\partial_\xi^2 X(t,\xi)+\partial_\xi (X^2(t,\xi)))dt+dW(t,\xi),\; t>0,\; \xi \in (0,1),\\
X(t,0)&=&X(t,1)=0,\; t>0,\\
X(0,\xi)&=&x(\xi),\; \xi\in (0,1).
\end{array}\right. 
\end{equation}
The unknown $X$ is a real valued process depending on $\xi\in [0,1]$ and $t\ge 0$ and 
$dW/dt$ is a space-time white noise on $[0,1]\times [0,\infty)$. This equation has been studied 
by several authors (see \cite{bertini},  \cite{DaDeTe94}, 
\cite{DPG}, \cite{gyongy})
and it is known that there exists a unique solution with paths in 
$C([0,T];L^p(0,1))$ if the initial data $x\in L^p(0,1)$, $p\ge 2$.
In this article, we want to prove new properties on the transition semigroup associated to 
\eqref{e0.1}.

We rewrite \eqref{e0.1} as an abstract differential equation in the Hilbert space $H=L^2(0,1)$,
\begin{equation}
\label{e1}
\left\{\begin{array}{l}
dX=(AX+b(X))dt+dW_t,\\
\\
X(0)=x.
\end{array}\right. 
\end{equation}
As usual  $A=\partial _{\xi\xi}$ with Dirichlet boundary conditions, on the domain
$
D(A)=H^2(0,1)\cap H^1_0(0,1)$, $b(x)=\partial _{\xi}\,(x^2)$. Here and below, for $s\ge 0$, $H^s(0,1)$ is the standard $L^2(0,1)$ based Sobolev space.  Also, $W$ is a cylindrical Wiener 
process on $H$.
We denote by $X(t,x)$ the solution.

We denote by   $(P_t)_{t\ge 0}$ the transition semigroup   associated  to equation \eqref{e1} on $\mathcal B_b(H)$, the space of all real bounded and Borel functions on $H$ endowed with the norm
$$
\|\varphi\|_0=\sup_{x\in H}|\varphi(x)|,\quad \forall\;\varphi\in \mathcal B_b(H).
$$
We know that $P_t$   possesses a unique invariant measure $\nu$ so that  $P_t$ is uniquely extendible to  a strongly continuous semigroup of contractions on $L^2(H,\nu)$ (still denoted by $P_t$)  whose infinitesimal generator   we shall denote by $\mathcal L$. Let  $\mathcal E_A(H)$ be 
  the linear span of real parts of all $\varphi$ of the form
$$\varphi_h(x):=e^{i \langle h,x   \rangle},\;x\in D(A).
$$
We have proved in \cite{DaDe07}   that  $\mathcal E_A(H)$
is a core for $\mathcal L$ and that
\begin{equation}
\label{e2z}
\mathcal  L\varphi=\frac12\mbox{\rm Tr}\;[QD_x^2\varphi]+\langle Ax+b(x)  ,D_x\varphi\rangle,\quad \forall\;\varphi\in \mathcal E_A(H).
\end{equation}
 \medskip
Here and below, $D_x$ denotes the differential with respect to $x\in H$. When $\varphi$ is a real 
valued function, we often identify $D_x\varphi$ with its gradient. Similarly, $D^2_x$ is the second
differential and for a real valued function $D^2_x\varphi$ can be identify with the Hessian. 

\medskip

In this paper, we use a formula    for  $P_tD_x\varphi$  which depends on $\varphi$ but not on its derivative. To our knowledge, this formula is new. 

For a finite dimensional stochastic equation a formula for $P_tD_x$ can be obtained, under suitable assumptions, using the Malliavin calculus and it is the key tool  for proving the existence of a density of the law of $X(t,x)$ with respect to the Lebesgue measure, see \cite{Ma97}.  Concerning SPDEs, several results are available for   densities of finite dimensional projections of the law of the solutions, see \cite{Sa05} and the references therein. For these results, Malliavin calculus is used on a finite 
dimensional random variable. Malliavin calculus is difficult to generalize to a true infinite dimensional 
setting and it does not seem useful to give estimate on $P_tD_x\varphi$ in terms of $\varphi$. The 
formula we use allows a completely different approach. It relates $P_tD_x$ to $D_xP_t$. 
In the recent years several formulae for $D_xP_t\varphi$
independent on  $D_x\varphi$ have been proved  thanks to  suitable generalizations of the Bismut--Elworthy--Li
 formula (BEL). Thus, combining our formula to estimates obtained on $D_xP_t$ implies useful 
 information on $P_tD_x$.
As we shall show, these can be used to extend to the measure $\nu$   a basic integration by parts identity well known for Gaussian measures.\medskip

Let us explain the main ideas. Let $u(t,x)=P_t\varphi(x)$, under suitable conditions, it is  a  solution of the Kolmogorov equation
\begin{equation}
\label{e3z}
\left\{\begin{array}{l}
\ds D_tu(t,x)= \frac12\mbox{\rm Tr}\;[QD_x^2u(t,x)]+\langle Ax+b(x)  ,D_x u(t,x)\rangle,\\
\\
u(0,x)=\varphi(x).
\end{array}\right.
\end{equation}
Set $v_h(t,x)=D_xu(t,x)\cdot h$, the differential of $u$ with respect to $x$ in the direction $h$. Then formally $v_h$ is a solution to the equation
\begin{equation}
\label{e4z}
\left\{\begin{array}{l}
\ds D_tv_h(t,x)= \mathcal L v_h(t,x) +\langle Ah+b'(x)h  ,D_x u(t,x)\rangle,\\
\\
v_h(0,x)=D_x\varphi(x)\cdot h.
\end{array}\right.
\end{equation}
 We notice that formal computations may be made rigorous by approximating   $\varphi$ with elements from $\mathcal E_A(H)$. By \eqref{e4z} and variation of constants, it follows that
\begin{equation}
\label{e7z}
 v_h(t,x)=P_t(D_x\varphi(x)\cdot h)  + \int_0^tP_{t-s}(\langle Ah+b'(x)h  ,D_x u(s,x)\rangle)ds,
\end{equation}
which implies
\begin{equation}
\label{e8z}
  P_t(D_x\varphi(x)\cdot h)= D_xP_t\varphi(x)\cdot h-   \int_0^tP_{t-s}(\langle Ah+b'(x)h  ,D_x u(s,x)\rangle)ds.
\end{equation}
This formula allows to obtain the following estimate:
For all  $\varphi\in \mathcal B_b(H)$,     $\delta>0$ and all $h\in H^{{1+\delta}}(0,1)$, we have
\begin{equation}
\label{e2d}
|P_t(D\varphi\cdot h)(x)|\le ce^{ct}(1+t^{-1/2})(1+|x|_{L^4})^{8}\,\|\varphi\|_{0}\,|h|_{1+\delta},
\end{equation}
where $|\cdot|_{1+\delta}$ is the norm in $H^{1+\delta}(0,1)$. We will not prove this formula here, it could be proved by similar arguments as in section \ref{s3}. 

Integrating with respect to $\nu$ over $H$ and taking into account   the invariance of $\nu$, yields
\begin{equation}
\label{e9}
\begin{array}{lll}
 \ds \int_H (D_x\varphi(x)\cdot h)d\nu&=&\ds \int_H(D_xP_t\varphi(x)\cdot h)\,d\nu\\
 \\
 &&- \ds  \int_0^t\int_H (\langle Ah+b'(x)h  ,D_xP_s\varphi(x) \rangle)d\nu\, ds.
 \end{array}
\end{equation}
Using identity \eqref{e9} we arrive at  the main result of the paper, proved in Section 2.
\begin{Theorem}
\label{t1}
For any $p>1$, $\delta>0$, there exists $C>0$ such that for all  $\varphi\in  \mathcal B_b(H)$ and all $h\in H^{1+\delta}(0,1)$, we have
\begin{equation}
\label{e2d}
\int _H D_x\varphi(x)\cdot h \, \nu(dx) \le C\|\varphi\|_{L^p(H,\nu)} \,|h|_{1+\delta},
\end{equation}
 for $t>0$, where $|\cdot|_{1+\delta}$ is the norm in $H^{1+\delta}(0,1)$.  \end {Theorem}

For a gaussian measure, it is easy to obtain such estimate. In fact, if $\mu$ is the invariant measure 
of the stochastic heat equation on $(0,1)$, {\it i.e.}  equation \eqref{e1} without the nonlinear term, then 
the same formula holds with $\delta=0$. Thus our result is not totally optimal and we except that it
can be extended to $\delta=0$.

Also, identity \eqref{e9} is general and we believe that it can be used in many other situations. For instance, we will investigate the generalization of our results to other SPDEs such as reaction--diffusion and 2D- Navier--Stokes equations.  This will be the object of a future work.

In section \ref{s2}, we show that Theorem \ref{t1} can be used to derive an integration by part
formula for the measure $\nu$. Theorem \ref{t1} is proved in section \ref{s3}. 
 
 We end this section with some notations.
 We shall denote by $(e_k)$ an orthonormal basis in $H$ and    by $(\alpha_k)$ a sequence of positive numbers such that
 $$
 Ae_k=-\alpha_k e_k,\quad k\in\N.
 $$
 For any $k\in\N$,  $D_k$ will represent the directional derivative in the direction of $e_k$.

The norm of $L^2(0,1)$ is denoted by $|\cdot|$. For $p\ge 1$, $|\cdot|_{L^p}$ is  the norm of $L^p(0,1)$. 
The operator $A$  is self--adjoint negative. For any $\alpha\in\R$, $(-A)^\alpha$ denotes the 
$\alpha$ power of the operator $-A$ and   $|\cdot|_\alpha$ is the norm of $D((-A)^{\alpha/2})$ which is equivalent to the norm of the Sobolev space $H^\alpha(0,1)$. We have
$|\cdot|_0=|\cdot|=|\cdot|_{L^2}$. We shall use the interpolatory estimate
\begin{equation}
\label{e1a}
|x|_\beta\le |x|^{\frac{\gamma-\beta}{\gamma-\alpha}}_\alpha\;|x|^{\frac{\beta-\alpha}{\gamma-\alpha}}_\gamma,\quad \alpha<\beta<\gamma,
\end{equation}
and the Agmon's inequality
\begin{equation}
\label{e2a}
|x|_{L^\infty}\le |x|^{\frac12}\,|x|^{\frac12}_1.
\end{equation}

\subsection*{Acknowledgement} A. Debussche is partially supported by the french government thanks to the
ANR program Stosymap. He also benefits from the support of the french government ``Investissements d'Avenir" program ANR-11-LABX-0020-01. 

\section{Integration by part formula for $\nu$} 
 \label{s2}
 
  This section is devoted to some consequences of Theorem \ref{t1}. Here we take $p=2$ for simplicity. In this case \eqref{e2d} can be rewritten as
\begin{equation}
\label{e2dd}
\int _H ((-A)^{-\alpha}D_x\varphi(x)\cdot h) \, \nu(dx) \le C\|\varphi\|_{L^2(H,\nu)} \,|h|,\quad\forall\;h\in H,
\end{equation}
where $\alpha=\frac{1+\delta}2$.
\begin{Proposition}
\label{p1e}
Let $\alpha >\frac12$, then for any $h\in H$  the linear operator
$$
\varphi\in C^1_b(H)\mapsto ((-A)^{-\alpha} D_x\varphi(x)\cdot h) \in C_b(H)$$
is closable in $L^2(H,\nu).$
\end{Proposition}
\begin{proof}
Let $(\varphi_n)\subset   C^1_b(H)$ and $f\in L^2(H,\nu)$ such that
$$
\varphi_n\to 0\;\;\mbox{\rm in}\;L^2(H,\nu),\quad 
((-A)^{-\alpha} D_x\varphi_n(x)\cdot h)\to f\;\;\mbox{\rm in}\;L^2(H,\nu).
$$
Let $\psi\in C^1_b(H)$,
then by \eqref{e2dd} it follows that
$$
\begin{array}{l}
\ds\left|\int_H   [\psi(x)((-A)^{-\alpha}D_x\varphi_n(x) \cdot h)+\varphi_n(x)((-A)^{-\alpha}D_x\psi(x)\cdot h)]\,\nu(dx)\right|\\
\\
\ds \le   \|\varphi_n\psi\|_{L^2(H,\nu)}\,|h|_H
\le  \|\psi\|_0\;  \|\varphi_n\|_{L^2(H,\nu)}\,|h|_H.
\end{array}
$$
Letting $n\to\infty$, yields
$$
 \int_H   \psi(x) f(x)\,\nu(dx)=0,
$$
which yields $f=0$ by the arbitrariness of $\psi$.
\end{proof}
We can now define the Sobolev space $W^{1,2}_\alpha(H,\nu).$ First we improve Proposition \ref{p1e}.
\begin{Corollary}
\label{c2e}
Let $\alpha >\frac12$, then  the
 linear operator
$$
\varphi\in C^1_b(H)\mapsto (-A)^{-\alpha} D_x\varphi\in C_b(H;H)$$
is closable in $L^2(H,\nu).$
\end{Corollary}
\begin{proof}
By Proposition \ref{p1e} taking $h=e_k$ we see that $D_k
$ is a closed operator on $L^2(H,\nu)$ for any $k\in \N$. Set
$$
(-A)^{-\alpha} D_x\varphi(x)=\sum_{k=1}^\infty \alpha_k^{-\alpha}D_k\varphi(x)\ e_k,\quad \forall\;h\in H,
$$
the series being convergent in $L^2(H,\nu)$.
Then
$$
|(-A)^{-\alpha} D_x\varphi(x)|^2=\sum_{k=1}^\infty \alpha_k^{-2\alpha}|D_k\varphi(x)|^2
$$
  Let $(\varphi_n)\subset   C^1_b(H)$ and $F\in L^2(H,\nu;H)$ such that
$$
\varphi_n\to 0\;\;\mbox{\rm in}\;L^2(H,\nu),\quad 
(-A)^{-\alpha} D_x\varphi \to F\;\;\mbox{\rm in}\;L^2(H,\nu;H).
$$
We have to show that $F=0$.

Now for any $k\in\N$ we have $D_k\varphi_n(x)\to \alpha_k^\alpha\langle F(x),e_k\rangle$ in $L^2(H,\nu)$. So, 
 $\langle F,e_k\rangle=0$ and the conclusion follows.
 
\end{proof}
Let us   denote by $W^{1,2}_\alpha(H,\nu)$ the domain of the closure of $(-A)^{-\alpha} D_x$.  Then if  $M^*$ denotes the adjoint of $(-A)^{-\alpha} D_x$ we have
 \begin{equation}
\label{e11}
\int_H ((-A)^{-\alpha} D_x\varphi(x)\cdot F(x))\,\nu(dx)=\int_H \varphi(x)\,M^*(F)(x)\,\nu(dx).
\end{equation}
Set now $F_h(x)=h$ where $h\in H$.
By Theorem 1  $F_h$ belongs  to the domain of $M^*$. Setting $M^*(F_h)=v_h$
we obtain the following integration by part formula.
\begin{Proposition}
Let $\alpha >\frac12$, then for any $h\in H$ there exists a function $v_h\in L^2(H,\nu)$ such that
 \begin{equation}
\label{e12}
\int_H ((-A)^{-\alpha} D_x\varphi(x) \cdot h)\,\nu(dx)=\int_H \varphi(x)\,v_h(x)\,\nu(dx),
\end{equation}
for any $\varphi\in W^{1,2}_\alpha(H,\nu)$.
\end{Proposition}
 By \eqref{e12} it follows that the measure $\nu$ possesses the  Fomin derivative in all directions $(-A)^{-\alpha}h$ for $h\in H$, see e.g. \cite{Pu98}.

If, in \eqref{e1}, $b=0$ then the gaussian measure $\mu=N_Q$,  where $Q=-\frac12\;A^{-1}$, is the
invariant measure  and  
$
v_h(x)=\sqrt 2\langle Q^{-1/2}x,h\rangle$.
Then \eqref{e12} reduces to the usual integration by parts formula for the Gaussian measure $\mu$. 
Note that it follows that, as already mentionned, Theorem \ref{t1} is true with $\delta=0$ in this case.

 We recall the importance of   formula \eqref{e12}  for different topics as Malliavin calculus \cite{Ma97}, definition of integral on infinite dimensional surfaces of $H$ \cite{AiMa88},  \cite{FePr92}, \cite{Bo98},  definition   of BV functions in abstract Wiener spaces \cite{AmMiMaPa10}, infinite dimensional generalization of DiPerna-Lions  theory \cite{AmFi09}, \cite{DaFlRo14} and so on.

We think that Theorem \ref{t1} open the possibility to study these topics in the more general situations of non  Gaussian measures.\bigskip

 \section{Proof of Theorem \ref{t1}}
\label{s3}

  For $h\in H$, $\eta^h(t,x)$ is the differential of $X(t,x)$ in  the direction $h$ and $(\eta^h(t,x))_{\ge 0}$ satisfies the equation
 \begin{equation}
\label{e2}
\left\{\begin{array}{l}
\ds\frac{d\eta^h(t,x)}{dt}=A\eta^h(t,x)+b'(X(t,x))\eta^h(t,x),\\
\\
\eta^h(0,x)=h.
\end{array}\right. 
\end{equation}
Note that this equation as well as the computations below are done at a formal level. They could easily 
be justified rigorously by an approximation argument, such as Galerkin approximation for instance.
The following result is proved in \cite{DaDe07}, see Proposition 2.2.
\begin{Lemma}
\label{l1}
For any $\alpha\in[-1,0]$ there exists $c=c(\alpha)>0$ such that
\begin{equation}
\label{e3}
 e^{-c\int_0^t|X(s,x)|_{L^4}^{\frac83}\,ds}\,|\eta^h(s,x)|^2_\alpha+\int_0^t
 e^{-c\int_s^t|X(\tau,x)|_{L^4}^{\frac83}\,d\tau}\,|\eta^h(s,x)|^2_{1+\alpha}\,ds\le |h|^2_\alpha.
\end{equation}
\end{Lemma}
We introduce the following Feynman-Kac  semigroup
$$
S_t\varphi(x)=\E \left[\varphi(X(t,x))e^{-K\int_0^t|X(s,x)|^4_{L^4}\,ds}   \right]
$$
Next lemma is a slight generalization of  Lemma 3.2 in \cite{DaDe07}.
\begin{Lemma}
\label{l2}
For any $\alpha\in[0,1]$ and $p>1$, if $K$ is chosen large enough then for any $\varphi$ Borel and bounded
we have
\begin{equation}
\label{e4}
 |D_xS_t\varphi(x)|_{\alpha}\le c\, e^{ct}(1+t^{-\frac{1+\alpha}{2}})(1+|x|_{L^6}^3)\,\left[\E\left( \varphi^p(X(t,x)\right)\right]^{1/p},
\end{equation}
where $c$ depends on $p,\, K,\, \alpha$.
\end{Lemma}
\begin{proof} It is clearly suficient to prove the result for $p\le 2$. 
We proceed as in \cite{DaDe07} and write
$$
 D_xS_t\varphi(x)\cdot h = I_1+I_2.
$$
where
$$
I_1=\frac1t \E\left(e^{-K\int_0^t|X(s,x)|^4_{L^4}\,ds}\varphi(X(t,x)) \int_0^t (\eta^h(s,x),dW(s) )     \right)   
$$
and
$$
I_2=-4K \E\left(e^{-K\int_0^t|X(s,x)|^4_{L^4}\,ds}\varphi(X(t,x)) \int_0^t(X^3(s,x),\eta^h(s,x))ds     \right).
$$
For $I_1$ we have with $\frac1p+\frac1q=1$:
$$
\begin{array}{lll}
I_1&\le& \ds\frac1t\left[\E\left( \varphi^p(X(t,x)\right)\right]^{1/p} \left[\E\left(e^{-Kq \int_0^t|X(s,x)|^4_{L^4}\,ds}\left|\int_0^t (\eta^h(s,x),dW(s)\right|^q    \right)   \right]^{1/q}
\end{array} 
$$
Using It\^o's formula for $|z(t)|^q=e^{-Kq \int_0^t|X(s,x)|^4_{L^4}\,ds}\left|\int_0^t (\eta^h(s,x),dW(s))\right|^q  $, we get:
$$
\begin{array}{ll}
\ds |z(t)|^q& \ds =-4Kq \int_0^t |X(s,x)|^4_{L^4} |z(s)|^q\,ds\\
&\ds + q\int_0^t e^{-K \int_0^s|X(s,x)|^4_{L^4}\,ds}|z(s)|^{q-2} z(s) (\eta^h(s,x),dW(s))\\
&\ds + \frac12q(q-1)\int_0^t e^{-2K \int_0^t|X(s,x)|^4_{L^4}\,ds}|z(s)|^{q-2} |\eta^h(s,x)|^2ds.
\end{array}
$$
We deduce:
$$
\begin{array}{ll}
\ds\E\left(\sup_{r\in [0,t]} |z(r)|^q    \right)
&\ds \le  q\E\left(\sup_{r\in[0,t]} \left|\int_0^r e^{-K \int_0^s|X(s,x)|^4_{L^4}\,ds}|z(s)|^{q-2} z(s) (\eta^h(s,x),dW(s))\right|\right)\\
&\ds + \frac12q(q-1)\E\left(\int_0^t e^{-2K \int_0^t|X(s,x)|^4_{L^4}\,ds}|z(s)|^{q-2} |\eta^h(s,x)|^2ds
\right)\\
&=A_1+A_2.
\end{array}
$$
By a standard martingale inequality, \eqref{e1a} and Lemma \ref{l1}, we have
$$
\begin{array}{ll}
\ds A_1& \le \ds 3q \E\left(\left|\int_0^t e^{-2K \int_0^s|X(s,x)|^4_{L^4}\,ds}|z(s)|^{2(q-1)} |\eta^h(s,x)|^2ds\right|^{1/2}\right)\\
&\ds \le 3q \E\left(\sup_{r\in [0,t]} |z(r)|^{q-1}\left(\int_0^t e^{-2K \int_0^s|X(s,x)|^4_{L^4}\,ds} |\eta^h(s,x)|^2ds\right)^{1/2}\right)\\
&\le \ds 3q \E\left(\sup_{r\in [0,t]} |z(r)|^{q-1} \left(\int_0^te^{-2K\int_0^s|X(s,x)|^4_{L^4}\,ds}\,|\eta^h(s,x)|^{2(1-\alpha)}_{-\alpha}\,|\eta^h(s,x)|^{2\alpha}_{1-\alpha}\,ds \right)^{1/2}   \right)  \\
&\ds \le 3q t^{\frac{1-\alpha}2}\,|h|_{-\alpha}\E\left(\sup_{r\in [0,t]} |z(r)|^{q-1} \right)\\
&\ds \le 3q t^{\frac{1-\alpha}2}\,|h|_{-\alpha}\left[\E\left(\sup_{r\in [0,t]} |z(r)|^{q} \right)\right]^{(q-1)/q}
\\
&\ds \le \frac14 \E\left(\sup_{r\in [0,t]} |z(r)|^{q}\right) + c  t^{\frac{q(1-\alpha)}2}\,|h|_{-\alpha}^q.
\end{array} 
$$
Similarly:
$$
\begin{array}{ll}
\ds A_2& \le \ds \frac12q(q-1) \E\left(\sup_{r\in [0,t]} |z(r)|^{q-2}\int_0^t e^{-2K \int_0^s|X(s,x)|^4_{L^4}\,ds} |\eta^h(s,x)|^2ds\right)\\
&\le \ds \frac12q(q-1) \E\left(\sup_{r\in [0,t]} |z(r)|^{q-2} \int_0^te^{-2K\int_0^t|X(s,x)|^4_{L^4}\,ds}\,|\eta^h(s,x)|^{2(1-\alpha)}_{-\alpha}\,|\eta^h(s,x)|^{2\alpha}_{1-\alpha}\,ds    \right)  \\
&\ds  \le\frac12q(q-1) t^{1-\alpha}\,|h|_{-\alpha}^2\E\left(\sup_{r\in [0,t]} |z(r)|^{q-2} \right)\\
&\ds \le\frac12q(q-1) t^{1-\alpha}\,|h|_{-\alpha}^2\left[\E\left(\sup_{r\in [0,t]} |z(r)|^{q} \right)\right]^{(q-2)/q}\\
&\ds \le \frac14 \E\left(\sup_{r\in [0,t]} |z(r)|^{q}\right) + c  t^{\frac{q(1-\alpha)}2}\,|h|_{-\alpha}^q.
\end{array} 
$$
We deduce:
$$
I_1\le c t^{-\frac{1+\alpha}2}\,|h|_{-\alpha} \left[\E\left( \varphi^p(X(t,x)\right)\right]^{1/p} 
$$

For $I_2$ we write
$$
\begin{array}{lll}
I_2&=&4K \E\left(e^{-K\int_0^t|X(s,x)|^4_{L^4}\,ds}\varphi(X(t,x)) \int_0^t(X^3(s,x),\eta^h(s,x))ds     \right)\\
&\le& \ds 4K\left[\E\left( \varphi^p(X(t,x)\right)\right]^{1/p}\left[\E\left(e^{-Kq\int_0^t|X(s,x)|^4_{L^4}\,ds}\left(\int_0^t|X(s,x)|^3_{L^6} \,|\eta^h(s,x)|\,ds\right)^q \right)   \right]^{1/q}.
 \end{array} 
$$
By Lemma \ref{l1} and Proposition 2.2 in \cite{DaDe07}
$$
I_2\le c_q(1+|x|_{L^6}^3)\left[\E\left( \varphi^p(X(t,x)\right)\right]^{1/p}\,|h|_{-1}.
$$
Gathering the estimates on $I_1$ and $I_2$ gives the result.
\end{proof}

\begin{Lemma}
\label{l3}
For any $\alpha\in [0,1)$, $p>1$, $q>1$ satisfying $\frac1p+\frac1q<1$, if $K$ is chosen large enough  there exists a constants $c$ depending on $\alpha,\, p,\, q$ such that for any $\varphi$ Borel bounded
and $h\,:\, H\to D((-A)^{-\alpha/2})$ Borel such that $\int_H  |h(x)|_{-\alpha}^q \nu(dx)<\infty$
 we have
\begin{equation}
\label{e5}
\left|\int_H D_xP_t\varphi(x)\cdot h(x) \nu(dx) \right|\le c e^{ct} (1+t^{-\frac{1+\alpha}2})\|\varphi\|_{L^p(H,\nu)}\left(\int_H  |h(x)|_{-\alpha}^q \nu(dx)
\right)^{1/q}.
\end{equation}
\end{Lemma}
\begin{proof}
We first prove a similar estimate for $S_t$. Using Lemma \ref{l2} we have by H\"older inequality
$$
\begin{array}{l}
\ds \left|\int_H D_xS_t\varphi(x)\cdot h(x) \,\nu(dx)\right| \\
\le \ds c\, e^{ct}(1+t^{-\frac{1+\alpha}{2}})\int_H (1+|x|_{L^6}^3)\,\left[\E\left( \varphi^p(X(t,x)\right)\right]^{1/p} |h(x)|_{-\alpha} \nu(dx)\\
\ds \le c\, e^{ct}(1+t^{-\frac{1+\alpha}{2}})\left[ \int_H (1+|x|_{L^6}^3)^r \nu(dx)\right]^{1/r}\\
\\
\ds\times \left[ \int_H \E\left( \varphi^p(X(t,x)\right)\nu(dx) \right]^{1/p} 
\left[ \int_H |h(x)|_{-\alpha}^q \nu(dx)\right]^{1/q}.
\end{array}
$$
with $\frac1p+\frac1q+\frac1r=1$. Thus by Proposition 2.3 in \cite{DaDe07} and the invariance of $\nu$:
$$
\left|\int_H D_xS_t\varphi(x)\cdot h(x) \, \nu(dx)\right| \le c e^{ct} (1+t^{-\frac{1+\alpha}2})\|\varphi\|_{L^p(H,\nu)}\left(\int_H  |h(x)|_{-\alpha}^q \nu(dx)
\right)^{1/q}.
$$
We then proceed as in \cite{DaDe07} to get a similar estimate on $P_t$. We write
$$
P_t\varphi(x)=S_t\varphi(x)+K\int_0^tS_{t-s}(|x|_{L^4}^4P_s\varphi)ds.
$$
It follows that, using the estimate above with $p>\tilde p>1$ such  that $\frac1{\tilde p}+\frac1q<1$:
$$
\begin{array}{l}
\ds \left|\int_H D_xP_t\varphi(x)\cdot h(x) \nu(dx)\right| \le  ce^{ct} (1+t^{-\frac{1+\alpha}2})\|\varphi\|_{L^p(H,\nu)}\left(\int_H  |h(x)|_{-\alpha}^q \nu(dx)
\right)^{1/q}\\
\\
\ds+ K\int_0^tce^{c(t-s)} (1+(t-s)^{-\frac{1+\alpha}2})\left(\int_H  |x|_{L^4}^4 |P_s \varphi(x)|^{\tilde p} d\nu(dx)
\right)^{1/\tilde p}\\
\\
\ds\times  \left(\int_H  |h(x)|_{-\alpha}^q \nu(dx)\right)^{1/q}\, ds .
\end{array} 
$$
The result follows by H\"older inequality and the invariance of $\nu$.
\end{proof}

Theorem \ref{t1} follows directly from the following result thanks to the invariance of $\nu$ and taking for instance $t=1$.
\begin{Proposition}
For all  $p>1$, $\delta>0$, there exists a constant such that for $\varphi$ Borel bounded,      and all $h\in H^{{1+\delta}}(0,1)$, we have
\begin{equation}
\label{e8}
\left| \int_H P_t(D_x\varphi\cdot h)(x)\, \nu(dx)\right|\le ce^{ct}(1+t^{-1/2})\|\varphi\|_{L^p(H,\nu)}|h|_{1+\delta}
\end{equation}

 \end {Proposition}
 \begin{proof} By Poincar\'e inequality, it is no loss of generality to assume $\delta <\min\{2(1-\frac1p), \frac12\}$.

We start by integrating \eqref{e8z} on $H$: 
\begin{equation}
\label{e21}
\begin{array}{lll}
 \ds \int_H P_t(D_x\varphi\cdot h)(x)\nu(dx)&=&\ds \int_H(D_xP_t\varphi(x)\cdot h)\,\nu(dx)\\
 \\
 &&- \ds  \int_0^t\int_H P_{t-s}[(\langle Ah+b'(x)h  ,D_xP_s\varphi(x) \rangle)]ds \,\nu(dx).
 \end{array}
\end{equation}
 Then by Lemma \ref{l3} we deduce
$$
\begin{array}{ll}
\ds\left|\int_H P_t(D_x\varphi \cdot h)(x)\,  \nu(dx)\right| &\ds \le c e^{ct} (1+t^{-\frac{1}2})\|\varphi\|_{L^p(H,\nu)}|h|
\\
&\ds +\left|\int_H  \int_0^tP_{t-s}[(Ah+b'(\cdot)\cdot h,D_xP_{s}\varphi)]ds \, \nu(dx)\right|.
\end{array} 
$$
By the invariance of $\nu$:
$$
\int_H  \int_0^tP_{t-s}[(Ah+b'(\cdot)\cdot h,D_xP_{s}\varphi)]ds \, \nu(dx)=
\int_H  \int_0^t (Ah+b'(\cdot)\cdot h,D_xP_{s}\varphi)ds \, \nu(dx).
$$
Therefore, by Lemma \ref{l3} with $\alpha=1-\delta$ and $q=\frac2\delta$:
$$
\begin{array}{l}
\ds \left|\int_H  \int_0^tP_{t-s}[(Ah+b'(\cdot)\cdot h,D_xP_{s}\varphi)]ds \, \nu(dx)\right|\\
\ds \le \int_0^t c e^{c(t-s)} (1+s^{-1+\frac{\delta}2})\|\varphi\|_{L^p(H,\nu)}
\left(\int_H  |Ah+b'(\cdot)\cdot h|_{-1+\delta}^{\delta/2} \nu(dx)\right)^{\delta/2}.
\end{array}
$$
Note that
$$
\left| b'(x)\cdot h\right|_{-1+\delta} =\left|\partial_\xi \left( xh\right)\right|_{-1+\delta}\le c 
\left| xh\right|_{\delta}
$$
Then, we have:
$$
|xh|\le c|x|\,|h|_1
$$
by the embedding $H^1\subset L^\infty$ and
$$
|xh|_1\le c|x|_1\,|h|_1,
$$
since $H^1$ is an algebra.
We deduce by interpolation
$$
|xh|_\delta\le c|x|_\delta |h|_1.
$$
It follows
$$
\begin{array}{l}
\ds  \left|\int_H  \int_0^tP_{t-s}[(Ah+b'(\cdot)\cdot h,D_xP_{s}\varphi)]ds \, \nu(dx)\right|\\
\ds   \le c_\delta e^{ct} \|\varphi\|_{L^p(H,\nu)}\left(1+ \int_H  |x|_\delta^{\delta/2} \nu(dx)\right)^{2/\delta} |h|_{1+\delta}.
\end{array}
$$
We need to estimate $\int_H  |x|_\delta^{\delta/2} \nu(dx)$. 
We use the notation of \cite[Proposition 2.2]{DaDe07}
$$
\begin{array}{l}
|X(s,x)|_\delta\le |Y(s,x)|_\delta+|z_\alpha(s)|_\delta\\
\\
\le |Y(s,x)|^{1-\delta}\,|Y(s,x)|_1^\delta+|z_\alpha(s)|_\delta.
\end{array} 
$$
Using computation in \cite[Proposition 2.2]{DaDe07}, we obtain
$$
\sup_{t\in [0,1]} |Y(t,x)|^2 +\int_0^1 |Y(s,x)|^2_1 ds \le c(|x|^2 + \kappa) 
$$
where $\kappa$ is a random variable with all moments finite. 
It follows by \eqref{e1a}:
$$
\E\left( \int_0^1 |Y(s,x)|^{2/\delta}_\delta \, ds\right)
\le \E\left( \int_0^1 |Y(s,x)|^{2(1-\delta)/\delta} |Y(s,x)|^{2}_1 \, ds\right)\le c(|x|^2+1)^{1/\delta} 
$$
Genelarizing slightly Proposition 2.1 in \cite{DaDe07}, we have:
$$
\E\left(|z_\alpha(t)|_\delta^p\right)\le c_{\delta,p}
$$
for $t\in[0,1]$, $\delta<1/2$, $\alpha\ge 1$, $p\ge 1$.
We deduce:
$$
\E\left( \int_0^1 |X(s,x)|^{2/\delta}_\delta \, ds\right)\le c(|x|^2+1)^{1/\delta}.
$$
Integrating with respect to $\nu$ and using Proposition 2.3 in \cite{DaDe07} we deduce:
$$
\int_H  |x|_\delta^{\delta/2} \nu(dx)\le c_{\delta}
$$

Then \eqref{e8} follows.
\end{proof}

\end{document}